\documentclass[11pt]{article}

\usepackage[margin = 1in]{geometry}
\usepackage{setspace}

\usepackage[T1]{fontenc}
\usepackage{charter} 
\usepackage{float}
\usepackage{soul} 
\usepackage{bbm} 
\usepackage{verbatim}
\usepackage{amsmath}
\usepackage{amsthm}
\usepackage{amssymb}
\usepackage{textcomp}
\usepackage{graphicx}
\usepackage [english]{babel}
\usepackage{tikz}
\usepackage{dblfloatfix}

\usepackage{enumerate}

\usepackage[
backend=biber,
maxbibnames=99,
style=authoryear-comp,
sorting=nyt
]{biblatex}
\usepackage{doi}

\allowdisplaybreaks 

\usepackage{algorithm}
\usepackage{algpseudocode}

\newtheorem{theorem}{Theorem}[]

\newtheorem{lemma}[theorem]{Lemma}

\newcommand{\1}{\mathbbm{1}} 
\newcommand{\Prob}{\mathbbm{P}} 
\newcommand{\prob}{\mathbbm{P}} 
\newcommand{\E}{\mathbbm{E}} 
\newcommand{\ex}{\mathbbm{E}} 
\newcommand{\N}{\mathbbm{N}} 
\newcommand{\R}{\mathbbm{R}} 
\newcommand{\id}{\ensuremath{\operatorname{Id}}} 
\newcommand{\lcm}{\ensuremath{\operatorname{lcm}}} 
\newcommand{\bin}{\ensuremath{\operatorname{binomial}}} 

\title{\bf{An Algorithm to Recover Shredded Random Matrices}}
\author{
\bf{Caelan Atamanchuk}%
\footnote{\textit{McGill University}: {caelan.atamanchuk@gmail.com}}
\and
\bf{Luc Devroye}%
\footnote{\textit{McGill University}: {lucdevroye@gmail.com}}
\and
\bf{Massimo Vicenzo}%
\footnote{\textit{University of Waterloo}: {mvicenzo@uwaterloo.ca}}
}

\begin{document}

\setstretch{1.1}

\maketitle

\begin{abstract}
Given some binary matrix $M$, suppose we are presented with the collection of its rows and columns in independent arbitrary orderings. From this information, are we able to recover the unique original orderings and matrix? We present an algorithm that identifies whether there is a unique ordering associated with a set of rows and columns, and outputs either the unique correct orderings for the rows and columns, or the full collection of all valid orderings and valid matrices. We show that there is a constant $c > 0$ such that the algorithm terminates in $O(n^2)$ time with high probability and in expectation for random $n \times n$ binary matrices with i.i.d.\ entries $(m_{ij})_{ij=1}^n$ such that $\prob(m_{ij} = 1) = p$ and $\frac{c\log^2(n)}{n(\log\log(n))^2} \leq p \leq \frac{1}{2}$.
\end{abstract}

\section{Introduction}

In this work we study the problem of reconstructing a binary matrix after being "shredded". That is, we aim to explain when and how a matrix (in our case, drawn from a random model) can be uniquely reconstructed from just the information contained in the rows and columns without knowing how they are ordered. To give the setup more precisely, let $M = (m_{ij})_{i,j=1}^n$ be a $n \times n$ binary matrix with the rows and columns given labels in $[n] = \{1,...,n\}$ and let $\mathcal{C}(M) = \{\gamma_1 , ... , \gamma_n\}, \mathcal{R}(M) = \{\rho_1 , ... , \rho_n\}$ be the multisets of all the columns and rows of $M$ with some arbitrary ordering that is not necessarily that which they belong in, which we call these collections the shredded columns and rows. We say that $M$ is uniquely reconstructible (or just reconstructible) if there exists two unique permutations $\sigma = (\sigma_1,...,\sigma_n) $ and $\tau = (\tau_1,...,\tau_n)$ of $[n]$ such that 
\begin{equation}\label{eq:recon} 
\begin{bmatrix}
    \rho_{\sigma_1} \\ \vdots \\ \rho_{\sigma_n}
\end{bmatrix}
= 
\begin{bmatrix}
\gamma_{\tau_1} \ \cdots \ \gamma_{\tau_n}    
\end{bmatrix} .
\end{equation}
In particular, if a unique solution exists then both of the resulting matrices are equal to $M$ (there is always at least one solution equal to $M$, that being the correct, original ordering). If there are at least two pairs of permutations that satisfy (\ref{eq:recon}), then the matrix $M$ is not reconstructible and the collection of all pairs of permutations that satisfy the identities are the potential reconstructions of the original matrix. For example, the matrix that has every entry set to $0$ is not  uniquely reconstructible and has $(n!)^2$ solutions to (\ref{eq:recon}), even though each of the solutions corresponds to the same matrix. In fact, every matrix that has two equal rows (or columns) is not be reconstructible. Suppose that rows $\rho_i$ and $\rho_j$ are equal, $(\sigma,\tau)$ are a pair of permutations that satisfy (\ref{eq:recon}), and $\lambda$ is the transposition $(ij)$. Then, the pair $(\lambda\circ\sigma,\tau)$ also satisfies (\ref{eq:recon}) and so the matrix is not reconstructible.

To offer an analogy, we can view $M$ as a square binary picture. The problem is that of rebuilding the picture given the strips that come out after we send one copy of $M$ through a paper shredder upright and one copy sideways. If two strips are exactly the same, we do not know which spots to place the two strips, and we conclude that the picture is not reconstructible. However, for the algorithm, this notion of reconstructibility is not of much importance as all potential reconstructions are outputted.

In this work, we consider a matrix $M$ that has i.i.d.\ entries $m_{ij}$ with $\Prob(m_{i,j} = 1) = p$ and $\Prob(m_{i,j}=0) = 1-p$ for some $p$ that we view as a function of $n$. Since the 1's and 0's are essentially just labels in our model, there is a natural symmetry around $p=\frac{1}{2}$, and thus we assume throughout that $p \leq \frac{1}{2}$ for our analysis. 

When $p \geq \frac{(1+\epsilon)\log(n)}{n}$ for some $\epsilon > 0$ the matrix $M$ has pairwise distinct rows and columns with high probability (see Lemma \ref{duplicate}). Hence, for $p$ above that threshold, our definition of reconstructibility where matrices that have equal rows or columns are not reconstructible is with high probability equivalent to the simplified version where a matrix $M$ is reconstructible if 

$$\forall M', \ \bigg[\mathcal{C}(M) = \mathcal{C}(M') \ \text{and} \ \mathcal{R}(M) = \mathcal{R}(M') \implies M = M' \bigg] .$$

If $M$ is viewed as the adjacency matrix for some random directed graph on $n$ vertices (one with loops allowed), the columns $\gamma_1,...,\gamma_n$ represent the collection of all 1-in-neighbourhoods with only the central vertex's labelled removed and $\rho_1,...,\rho_n$ represent the collection of all 1-out-neighbourhoods with only the central vertex's labelled removed (removing the labels is the same as permuting them into some arbitrary labelling). By 1-out and 1-in neighbourhood of a vertex $v$, we mean the subgraph that contains all edges out of or into $v$ and the other vertices these edges are incident to. This observation links the work being done here to much of the related work and motivations in section 2.

The paper is organized as follows: In section 2 we discuss related work and some of the motivations behind work in the area of reconstruction problems. In section 3 we present the reconstruction algorithm along with our main result, and in Section 4 we prove the result. In section 5 we offer and prove a result concerning when matrices can be reconstructed. Finally, section 6 houses proofs for the lemmas that are used in the preceding sections.

\section{Related Work and Motivation}

Combinatorial reconstruction problems arise naturally in a number of pure and applied settings. The largest inspiration for such exploration comes from the reconstruction conjecture in combinatorics (see \textcite{Ulam1960,Kelly1957,Haray1974,Haray1985}): any graph $G$ on at least three vertices is reconstructible from the multiset of isomorphism classes of all the vertex-deleted subgraphs of $G$, often called the deck or $G$ and labelled $D(G)$ (the vertex deleted subgraphs of $G$ are all the induced subgraph obtained through deleting exactly one of the vertices of $G$). To be more exact, the conjecture states that for all graphs $G$ and $H$ on at least three vertices, $G$ is isomorphic to $H$ if and only if $D(G) = D(H)$. The use of random models has been vital in the study of this conjecture, with one important result coming from \textcite{Bollobas1990} who proved that as $n \to \infty$, an Erd\H{o}s-Reny\'i random graph with $\frac{c\log(n)}{n} \leq p \leq 1 - \frac{c\log(n)}{n}$ is uniquely reconstructible from a collection of only three of the vertex deleted subgraphs for any $c > \frac{5}{2}$. In particular, this means that with high probability, for appropriate choices of $p$, there is a subset $\{G_1,G_2,G_3\} \subseteq D(G)$ of three subgraphs such that for any other graph $H$, if $\{G_1,G_2,G_3\} \subseteq D(H)$, then $H$ is isomorphic to $G$. Before the Bollob\'as result, \textcite{Muller1976} had previously explored the reconstructibility of random graphs from the whole deck.

One interesting abstraction of the reconstruction conjecture related to the random pictures model is the new digraph reconstruction conjecture. Let $G$ and $H$ be two directed graphs and suppose that there is a bijection $f:V(G) \to V(H)$ such that $G \setminus v$ is isomorphic to $H \setminus f(v)$ for all $v \in V(G)$. Further suppose that the in-degrees and out-degrees of $v$ and $f(v)$ match for all $v \in V(G)$. Then, $G$ and $H$ must be isomorphic. The answer to this problem remains open. See \textcite{Ramachandran2004,Ramachandran1981} for discussion of the problem and the families of graphs for which the conjecture has been proven to be true.

Recently, extensive work has gone into studying the shotgun assembly problem for graphs. Introduced by \textcite{Mossel2019}, the problem asks how large must $r$ be so that a graph, commonly drawn from some random model, is uniquely determined by its collection of distance $r$-neighbourhoods around each vertex $v \in V(G)$ (by a distance $r$-neighbourhood of $v$ we mean the subgraph $N_r(v)$ that is induced by all vertices of graph distance at most $r$ from $v$). They consider both labelled and unlabelled versions of the problem. This topic has been studied for a variety of random models including Erd\H{o}s-Reny\'i graphs, random regular graphs, and simplicial complexes (for examples, see \textcite{Ding2022,Adhikari2022,Gaudio2022,Johnston2023,Huang2022,Mossel2015}). There has also been work put towards shotgun assembly problems in different contexts such as reconstructing random vertex colourings from $r$-neighbourhoods as seen in \textcite{Pryzykucki2022,Mossel2019,Liu2022}. 

In a similar vein, there is the problem of canonically labelling graphs and random graphs, and its main application in checking graph isomorphisms (early work in the topic can be seen in \textcite{Babai1980, Babai1983, BabaiErdos1980}). An algorithm which canonically labels a graph $G$, assigns the labels $1,2,\dots, n$ to the $n$ vertices of $G$ such that if $G$ is isomorphic to some graph $H$, then both should be given the same labelling by the algorithm. Of particular note to us are the results on canonically labelling the Erd\H{o}s-Reny\'i graph using only the $r$-neighbourhoods of each vertex. \textcite{Mossel2019} showed it is possible to canonically label a graph $G \sim G(n,p_n)$ when $np = \omega(\log^2(n))$ with using only 2 neighbourhoods. On the other hand, \textcite{gaudio2022local} showed for $np = o(\log^2(n)/(\log\log(n))^3)$ there are multiple isomorphic 2-neighbourhoods with high probability, which inhibits us from creating a canonical labelling. 

Another model that has received some attention is that of reconstructing random jigsaw puzzles. Once again introduced by \textcite{Mossel2019}, in this problem we are given the collection of vertices in a lattice with coloured half-edges drawn from some collection of $q$ colours. The problem asks how large $q$ must be so that with high probability the puzzle can be constructed into a complete picture from the collection of vertices and their coloured half-edges. Some work concerning this problem can be found in \textcite{Balister2019,Nenadov2017,Martinsson2016,Martinsson2019}.

The topic of this paper, reconstructing random matrices, has been studied before from another point of view. In \textcite{Narayanan2023}, the complete multiset of all $(n-k)^2$ $k \times k$ sub-matrices is given as the information to reconstruct with.

There is no lack of motivation from other sciences for studying reconstruction problems, such as the problem of \textsc{dna} shotgun sequencing. In shotgun assembly, the long \textsc{dna} strands are ``shotgunned'' into smaller pieces that are sequenced. From here, a reconstruction algorithm is used to infer what the original long strand was. For a probabilistic analysis of the unique reconstructibility of \textsc{dna} sequences from shotgunnned strands see \textcite{Motahari2013,Dyer1994,Arratia1996}. Note that the models here are what one of the shotgun assembly problems from \textcite{Mossel2015} is based on, with the special case of the path on $n$ vertices being studied. Shotgun assembly has also begun to appear in neural network theory. \textcite{Soudry2015} consider the problem of reconstructing large neural networks from smaller sub-networks.

\section{The Reconstruction Algorithm}

For a vector $x= (x_1, x_2, \dots, x_n) \in \{0,1\}^n$, we call $|x| = \sum_{i=1}^n x_i$ the weight or Hamming weight of $x$. If $S\subset [n]$ is a set of indices, then $\sum_{i\in S} x_i$ is the sub-weight of $x$ on $S$. Alternatively, the weight of $x$ can be seen as the number of 1's which appear in the entire vector, and the sub-weight in $S$ is the number of 1's in the vector $x$ restricted to the positions indicated by $S$. We have two algorithmic problems to solve:

\begin{enumerate}
    \item Find any permutation pair $(\sigma,\tau)$ that satisfies (\ref{eq:recon}).
    \item Find all permutation pairs $(\sigma,\tau)$ that satisfy (\ref{eq:recon}).
\end{enumerate}

Our algorithm solves (ii) and hence also (i). It can be broken down into two main parts: First we partition each row $\rho_i$ into sub-strings and compute the vector of the associated sub-weights for all $i \in [n]$. Then, using a trie, we can efficiently identify each $\rho_i$ with a position by matching these sub-weight vectors. If we are able to identify each $\rho_i$ with a unique position, then the algorithm is complete. We show this happens with high probability. 

In the case where this does not occur, we move on to part two of the algorithm, where we iterate through all possible permutations of the rows and check if the matrix is correct by checking if it contains all of the columns in $\mathcal{C}(M)$ with the correct multiplicities. Using the information gained from part one, we are able to reduce our search space from all $n!$ permutations of the rows, to a collection that has expected size $O(1)$.

\subsection{Part One}

Given the collection of unordered columns $\gamma_1,...,\gamma_n$, we create a Hamming weight partition of the columns $\mathcal{P} = (\mathcal{P}_0, \dots, \mathcal{P}_n)$, where $\mathcal{P}_i = \{1 \leq j \leq n : |\gamma_j| = i\}$. Now for each $j\in [n]$, and for each integer $k \in [\lfloor np \rfloor, \lfloor np\rfloor + \lfloor\sqrt{np}\rfloor]$, we compute 

$$s_{j,k} = \sum_{i \in \mathcal{P}_k} \gamma_{ij}, \quad \text{where} \quad \gamma_i = 
\begin{bmatrix}
\gamma_{i,1} \\ \vdots \\ \gamma_{i,n}    
\end{bmatrix}
.
$$

For a row to be able to be put in position $j$, its sub-weight on $\mathcal{P}_k$ must be equal to $s_{j,k}$ for all $k \in [\lfloor np \rfloor, \lfloor np\rfloor + \lfloor\sqrt{np}\rfloor]$. Using the values $s_{j,k}$ we store every potential position $j \in [n]$ in the leaves of a trie using the vectors $S_j = (s_{j,\lfloor np \rfloor} \dots, s_{j, \lfloor np \rfloor + \lfloor \sqrt{np}\rfloor})$ as input, which we call the sub-weight vectors associated with position $j$. See \textcite{Knuth1998} for more information on tries and their uses. In our trie, we associate each input with a path. Therefore, it is possible that several paths coincide and that $S_j$ is not unique, i.e.,\ $|\{S_j : 1 \leq j \leq n\}| < n$.

From the collection of rows $\rho_1,...,\rho_n$, we can compute the weight of each column in the original matrix $M$ even without knowing the order, since the weight of a column is invariant under permutation of the rows. This allows us to determine which column positions have which weights. Let $\mathcal{I} = (I_0,I_1,\dots,I_n)$, where 
$$I_j = \{i\in[n] : \text{The column in position $i$ has weight $j$} \}.$$ Now, for all $j \in [n]$, and for each integer $k \in [\lfloor np \rfloor, \lfloor np\rfloor +\lfloor \sqrt{np}\rfloor]$, we compute $t_{j,k}$, which is the sub-weight of the row $\rho_j$ on the indices $I_k$. We collect all of them into a vector 
$$T_{j} = \big(t_{j,\lfloor np \rfloor} \dots, t_{j, \lfloor np \rfloor + \lfloor\sqrt{np}\rfloor}\big) , $$
which we call the signature of $\rho_j$. Since entries $S_j$ (the sub-weight vector of position $j$) and $T_j$ (the signature of $\rho_j$) are generated from the same information with only potentially incorrect labels on $T_j$, we know that
$$\{S_j : 1 \leq j \leq n\} = \{T_j : 1 \leq j \leq n\} . $$
It follows that if $|\{S_j : 1 \leq j \leq n\}| = n$, then are able to identify a unique permutation for each row: For each $j \in [n]$, we define $\sigma_j$ to be the unique $\ell \in [n]$ such that $S_j = T_\ell$. Once the rows have been placed we have reconstructed the matrix and the permutation $\tau$ on the unordered columns can be determined. We do this by first constructing a trie based on all of the columns $C_1,...,C_n$ in the reconstructed matrix $M$ (these are the columns in their original, pre-shredded positions). If the trie has $n$ distinct leaves, then we can define a permutation $\tau$ for $\gamma_1,...,\gamma_n$ in the following way: For each $j \in [n]$, define $\tau_j$ to be the unique $\ell \in [n]$ such that $\gamma_{j} = C_\ell$. If either of the two tries do not have distinct leaves we move on to part 2. 

\subsection{Part Two} 

There are two possible cases where we end up requiring part two to complete the algorithm. First, we require part two when there is at least one leaf in the trie containing row sub-weight vectors which coincides with multiple rows, i.e. $|\{S_j : 1 \leq j \leq n\}| = L < n$. The second case where we require part two is when there are at least two columns coinciding with a single leaf in the trie containing the column vectors, i.e. $C_j=C_k$ for $j\neq k$.

For each vector $S_i \in \{S_j:1\leq j\leq n\}$, let $x_i$ be the multiplicity of that vector,  i.e. the number of rows $\rho_j$ where $S_j = S_i$. Then, since $\rho_j$ can only be assigned to a position $k$ such that $S_j = T_k$, there are $x_1!x_2!\dots x_L!$ possible permutations of $\rho_1,...,\rho_n$ that must be checked. For each possible permutation $\sigma$, we construct a matrix,

\begin{equation*}
    M' = 
    \begin{bmatrix}
    \rho_{\sigma_1} \\ \vdots \\ \rho_{\sigma_n}
    \end{bmatrix}
    .
\end{equation*}

Using the column trie, we determine if $\mathcal{C}(M) =\mathcal{C}(M')$. That is, we determine if both matrices contain the same set of columns with the same multiplicities. If this is the case, then $M'$ is a valid reconstruction. Let $\tau_j$ be an $\ell\in [n]$ such that the column in position $j$ in $M'$ is equal to $\gamma_\ell$ for all $j \in [n]$ (in particular choose the $\tau_j$ such that $\tau = (\tau_1,...,\tau_n)$ is a permutation). Note that at this point, $\ell$ need not be unique and so this could yield many valid matrices. The pair $(\sigma,\tau)$ permutes the rows and columns to create a valid reconstruction $M'$. Let $I_1, \dots, I_m$ be the sets of column indices ($|I_k|>1$) such that for every two indices $i,j \in I_k$, the columns $C_i, C_j$ in $M'$ are equal. Clearly the columns within each $I_{k}$ can be permuted and still give a valid $\tau$ for reconstructing. 

Therefore, for every valid $\sigma$ we compute one of the corresponding column permutations $\tau$, and the sets of indices $I_1, \dots, I_m$, and then output 

$$ (\sigma, \tau, S_{I_1} \times \dots \times S_{I_m}) . $$

Where $S_{I_k}$ is the group of permutations on the elements in the set $I_k$. The set of these triples can generate all of the pairs $(\sigma,\tau)$ which create a valid reconstruction. If we wish to retrieve every pair from the triple, we need only iterate over $\pi \in S_{I_1} \times \dots \times S_{I_k}$ and compute $(\sigma,\pi\tau)$.

\subsection{An Example}

Consider the following collection of its rows and columns (assume that $\mathcal{C}(M) = \{\gamma_1,\gamma_2,\gamma_3,\gamma_4\}$ and $\mathcal{R}(M) = \{\rho_1,\rho_2,\rho_3,\rho_4\}$ are ordered left to right and top to bottom respectively):
$$
\mathcal{C}(M) =
        \left(\!
        \begin{bmatrix}
            1\\ 0 \\ 1 \\ 1
        \end{bmatrix}\!\!,\!\!
        \begin{bmatrix}
            0 \\ 1 \\ 1 \\ 0
        \end{bmatrix}\!\!,\!\!
        \begin{bmatrix}
            0 \\ 1 \\ 0 \\ 1
        \end{bmatrix}\!\!,\!\!
        \begin{bmatrix}
            1 \\ 1 \\ 1 \\ 0
        \end{bmatrix}
        \!\right)
, \
\mathcal{R}(M) = 
        \left(
        \substack{
        \begin{bmatrix}
            1 & 0 & 1 & 0
        \end{bmatrix}, \\
        \begin{bmatrix}
            0 & 1 & 1 & 0
        \end{bmatrix}, \\
        \begin{bmatrix}
            0 & 1 & 1 & 1
        \end{bmatrix}, \\
        \begin{bmatrix}
            1 & 1 & 0 & 1
        \end{bmatrix}
        }
        \right) .
$$
We first construct the partition $\mathcal{P}$ from the column collection $\mathcal{C}(M)$. From this, for each position $j$, we compute the sub-weight vectors $S_j = (s_{j,2}, s_{j,3})$,
\begin{equation*}
    \mathcal{C}(M) =
        \left\{\!
        \begin{bmatrix}
            1 \\ 0 \\ 1 \\ 1
        \end{bmatrix}\!\!,\!\!
        \begin{bmatrix}
            0 \\ 1 \\ 1 \\ 0
        \end{bmatrix}\!\!,\!\!
        \begin{bmatrix}
            0 \\ 1 \\ 0 \\ 1
        \end{bmatrix}\!\!,\!\!
        \begin{bmatrix}
            1 \\ 1 \\ 1 \\ 0
        \end{bmatrix}
        \!\right\}
        \quad 
        \mathcal{P} = \!\left(
        \!\emptyset, 
        \emptyset, 
        \!\left\{ \begin{bmatrix}
            0 \\ 1 \\ 0 \\ 1
        \end{bmatrix}\!,\!
        \begin{bmatrix}
            0 \\ 1 \\ 1 \\ 0
        \end{bmatrix}
        \right\}\!,\!
        \left\{ \begin{bmatrix}
            1 \\ 1 \\ 1 \\ 0
        \end{bmatrix}\!,\!
        \begin{bmatrix}
            1 \\ 0 \\ 1 \\ 1
        \end{bmatrix}
        \right\}\!,
        \emptyset\!
        \right)
        \to 
        \begin{Bmatrix}
        (0,2), \\ (2,1), \\ (1,2), \\ (1,1)
        \end{Bmatrix} .
\end{equation*}
In this case, each of the $S_j$ are distinct, and so the trie we construct with them has exactly $n$ leaves. Each of the leaves contains the indices of the positions with sub-weight vectors which take them to said leaf (in bold):
$$
    
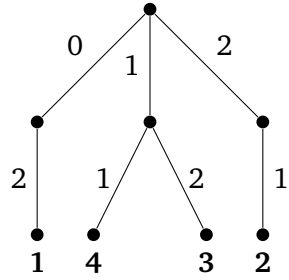
\centering
    \begin{tikzpicture} 
        \tikzstyle{n}=[fill=black, circle, inner sep=0.6mm]
        \node[n] {}
            child {
                node[n] {}
                child {
                    node[n,label=below:\textbf{1}] {}
                    edge from parent
                    node[left] {2}
                }    
                edge from parent
                node[above left] {0}}
            child {
                node[n] {}
                child {
                    node[n,label=below:\textbf{4}] {}
                    edge from parent
                    node[left] {1}
                    }
                child {
                    node[n,label=below:\textbf{3}] {}
                    edge from parent
                    node[right] {2}
                    }
                edge from parent
                node[left] {1}
            }
            child {
                node[n] {}
                child {
                    node[n,label=below: \textbf{2}] {}
                    edge from parent
                    node[right] {1}
                }
                edge from parent
                node[above right] {2}
                };
    \end{tikzpicture}
$$
Next we compute the signatures for each of the row vectors. We do this by first computing $\mathcal{I}$, and using the indices to determine the values of each entry,
\begin{equation*}
    \mathcal{R}(M) = 
        \left\{
        \substack{
        \begin{bmatrix}
            1 & 0 & 1 & 0
        \end{bmatrix}, \\
        \begin{bmatrix}
            0 & 1 & 1 & 0
        \end{bmatrix}, \\
        \begin{bmatrix}
            0 & 1 & 1 & 1
        \end{bmatrix}, \\
        \begin{bmatrix}
            1 & 1 & 0 & 1
        \end{bmatrix}
        }
        \right\} 
        \quad 
        \mathcal{I} = \!\left(
        \emptyset, \emptyset, 
        \left\{ 1,4 \right\},
        \left\{ 2,3 \right\},
        \emptyset
        \right)\to 
        \begin{Bmatrix}
        (1,1), \\ (0,2), \\ (1,2), \\ (2,1)
        \end{Bmatrix}.
\end{equation*}
Now we use the set of signatures and search through the trie generated by the sub-weight vectors. Each signature reaches a leaf, which then tells us which positions the that row are allowed to be in. In this example, they are each mapped to a unique position, telling us that $\sigma = (142)(3)$ is the permutation to apply on $\mathcal{R}(M)$ in order to obtain $M$. Doing so gives us our unique matrix
\begin{equation*}
    M = \begin{bmatrix}
        0 & 1 & 1 & 0 \\
        1 & 1 & 0 & 1 \\
        0 & 1 & 1 & 1 \\
        1 & 0 & 1 & 0
    \end{bmatrix} .
\end{equation*}
Since we have no duplicate columns, there is also a unique $\tau = (13)(24)$. The final output would be $((142)(3), (13)(24), \{\id\})$ as there is only one way to permute the columns to reconstruct the matrix.

For a second example, let us consider a case where we have duplicate sub-weight vectors and duplicate columns. Below is the result of doing part one to some matrix $M$, we can see that the first row in $\mathcal{R}(M)$ belongs in the second position, but the remaining rows' positions are unknown,

\begin{equation*}
    \mathcal{R}(M) = 
        \left\{
        \substack{
        \begin{bmatrix}
            1 & 0 & 0 & 0
        \end{bmatrix}, \\
        \begin{bmatrix}
            1 & 1 & 1 & 0
        \end{bmatrix}, \\
        \begin{bmatrix}
            0 & 1 & 1 & 1
        \end{bmatrix}, \\
        \begin{bmatrix}
            0 & 1 & 1 & 1
        \end{bmatrix}
        }
        \right\} \to 
        \begin{Bmatrix}
        (1,0), \\ (1,2), \\ (1,2), \\ (1,2)
        \end{Bmatrix}
        \quad
        \mathcal{C}(M) =
        \left\{\!
        \begin{bmatrix}
            1 \\ 0 \\ 0 \\ 1
        \end{bmatrix}\!\!,\!\!
        \begin{bmatrix}
            0 \\ 1 \\ 1 \\ 0
        \end{bmatrix}\!\!,\!\!
        \begin{bmatrix}
            1 \\ 0 \\ 1 \\ 1
        \end{bmatrix}\!\!,\!\!
        \begin{bmatrix}
            1 \\ 0 \\ 1 \\ 1
        \end{bmatrix}
        \!\right\}\to 
        \begin{Bmatrix}
        (1,2), \\ (1,0), \\ (1,2), \\ (1,2)
        \end{Bmatrix} .
\end{equation*}
As three rows have the same signature, we have $6$ permutations of the rows to check, 
$$\{(12), (12)(34), (123), (124), (1234), (1243)\},$$
which results in matrices,
\begin{align*}
    M_{(12)} = \begin{bmatrix}
        1 & 1 & 1 & 0 \\
        1 & 0 & 0 & 0 \\
        0 & 1 & 1 & 1 \\
        0 & 1 & 1 & 1
    \end{bmatrix} &\quad
    M_{(12)(34)} = \begin{bmatrix}
        1 & 1 & 1 & 0 \\
        1 & 0 & 0 & 0 \\
        0 & 1 & 1 & 1 \\
        0 & 1 & 1 & 1
    \end{bmatrix}  \\
    M_{(123)} = \begin{bmatrix}
        0 & 1 & 1 & 1 \\
        1 & 0 & 0 & 0 \\
        1 & 1 & 1 & 0 \\
        0 & 1 & 1 & 1
    \end{bmatrix} &\qquad
    M_{(124)} = \begin{bmatrix}
        0 & 1 & 1 & 1 \\
        1 & 0 & 0 & 0 \\
        0 & 1 & 1 & 1 \\
        1 & 1 & 1 & 0
    \end{bmatrix} \\
    M_{(1234)} = \begin{bmatrix}
        0 & 1 & 1 & 1 \\
        1 & 0 & 0 & 0 \\
        1 & 1 & 1 & 0 \\
        0 & 1 & 1 & 1
    \end{bmatrix} &\qquad \!\!
    M_{(1243)} = \begin{bmatrix}
        0 & 1 & 1 & 1 \\
        1 & 0 & 0 & 0 \\
        0 & 1 & 1 & 1 \\
        1 & 1 & 1 & 0
    \end{bmatrix} .
\end{align*}
Since there are duplicate rows, some of these permutations result in the same matrix. Regardless, using the column trie below, we can iterate through each $M_\sigma$ and see if $\mathcal{C}(M) = \mathcal{C}(M_\sigma)$:
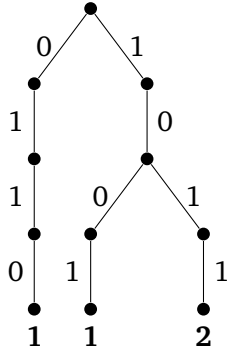
$$
    \centering
    \begin{tikzpicture}[level distance = 10mm]
    \tikzstyle{n}=[fill=black, circle, inner sep=0.6mm]
        \node[n] {}
            child {
                node[n] {}
                child {
                    node[n] {}
                    child {
                        node[n] {}
                        child {
                            node[n,label=below:\textbf{1}] {}
                            edge from parent
                            node[left] {0}
                        }
                        edge from parent
                        node[left] {1}  
                    }
                    edge from parent
                    node[left] {1}
                }
                edge from parent
                node[left] {0}  
            }
            child {
                node[n] {}
                child{
                        node[n] {}
                        child{
                            node[n] {}
                            child{
                                node[n,label=below:\textbf{1}] {}
                                edge from parent
                                node[left] {1}
                            }
                            edge from parent
                            node[left] {0}
                        }
                        child{
                            node[n] {}
                            child{
                                node[n,label=below:\textbf{2}] {}
                                edge from parent
                                node[right] {1}
                            }
                            edge from parent
                            node[right] {1}
                        }
                        edge from parent
                        node[right] {0}
                    }
                edge from parent
                node[right] {1}
            };
    \end{tikzpicture}
$$
From this we can see that the only $\sigma$ that give us valid matrices are from $(123)$ and $(1234)$, and since there are two identical columns in positions $2$ and $3$, the corresponding permutation groups are both $S_{\{2,3\}}$. The solution set for this example is
\begin{equation*}
    \{((123), (142), S_{\{2,3\}}),  ((1234),(142), S_{\{2,3\}})\}.
\end{equation*}

\subsection{Time Complexity} 

The time complexity achieved by our algorithm assumes the RAM model of computation. Computing the weights of the vectors and computing all the sub-weights takes time $O(n^2)$, since we can upper bound both of these by computing the sum of all entries in the matrix. Creating the trie with the sub-weight vectors would take time $O(n^{3/2})$ since the size of the strings used in the trie are bounded above by $\sqrt{np}\leq \sqrt{n}$. Since the height of the trie is $O(\sqrt{n})$, matching each $R_i$ to a set of positions at a leaf, takes total time $O(n^{3/2})$. Next we create the column trie, which takes $O(n^2)$ as we have $n$ length $n$ vectors to insert. It is interesting to note, the process of determining which rows belong in which positions is not the most time intensive step, in fact, simply determining the weights of the vectors is what gives us our time complexity.

In part two, for each valid permutation, we first check that $\mathcal{C}(M)$=$\mathcal{C}(M')$ by searching for each column in $M'$ in the column trie, keeping track of multiplicities. This takes time $O(n^2)$. Once a valid $\sigma$ is found, we must compute a single $\tau$, which we can get from reading the columns of $M'$ generated by $\sigma$ applied on the rows, in $O(n)$ time. Using the column trie, we can create the sets $I_1,\dots, I_m$ in $O(n^2)$ time. 

Let $P= x_1!x_2!\dots x_L!$ be the number of permutations $\sigma$ that we have to check. Then the entirety of part two takes expected time $O(n^2\E[P])$. In section 4, we show that $\ex[P] \to 1$ as $n \to \infty$ for $p$ in some range, implying that the expected time for the algorithm is $O(n^2)$. We also show that the probability we require step two to complete the algorithm tends to 0 as $n \to \infty$ for $p$ in another range, implying that the completion time is also $O(n^2)$ with high probability.

\section{Main Result}

The time complexity discussion from the previous section culminates in our main result. 

\begin{theorem}\label{alg}
If $p\geq \frac{16(1+\epsilon)\log^2(n)}{n (\log\log(n))^2}$ for $\epsilon > 0$, then, 
$$\Prob(\text{Algorithm terminates at first step}) \to 1 \text{ as } n \to \infty.$$
Hence, the algorithm succeeds in producing a unique reconstruction in $O(n^2)$ time with high probability. Furthermore, if $p\geq \frac{36(1+\epsilon)\log^2(n)}{n (\log\log(n))^2}$ for $\epsilon > 0$, the expected running time of the algorithm is also $O(n^2)$, with the expected number of permutations that require checking in step two converging to 1 as $n \to \infty$
\end{theorem}

To complete the proof of Theorem \ref{alg}, we need to bound the probability that two rows $\rho_i$ and $\rho_j$ share the same signature vectors $T_i$ and $T_j$. In order to analyze this we need to obtain some bounds on the size of each group in the partition $|\mathcal{P}| = (|\mathcal{P}_1|,...,|\mathcal{P}_n|)$. In particular, we want the groups near the average $np$ to be sufficiently large as these columns are the ones that the algorithm uses to generate sub-weight vectors and larger sub-strings produce sub-weights with larger variance. Since each column sum is a $\bin(n,p)$ random variable, and we have $n$ distinct columns, $|\mathcal{P}|$ has a multinomial distribution with parameters $n$ and $b = (b_{n,p,1},...,b_{n,p,n})$, where
$$b_{n,p,k} = \Prob(\bin(n,p) = k) = \binom{n}{k}p^k(1-p)^{n-k}.$$
The bounds we desire for $|\mathcal{P}|$ are given by the following lemma.

\begin{lemma}\label{chernoff}
Suppose that $p=p(n)$ is some sequence such that $np \geq 16$. There exists a positive constant $\gamma > 0$ such that $b_{n,p,\lfloor np \rfloor +i} \geq 2\gamma\frac{1}{\sqrt{np}}$ for all $i \in [0,\lfloor\sqrt{np}\rfloor]$. Furthermore,
$$\Prob\left(|\mathcal{P}_{\lfloor np \rfloor+i}| \leq \gamma\sqrt{\frac{n}{p}}\right) \leq e^{-\frac{1}{6}\gamma \sqrt{\frac{n}{p}}} . $$
\end{lemma}

Since the algorithm also requires passing to part two when two columns are equal, we need the next lemma as well.

\begin{lemma}\label{duplicate}
Let $M$ be an $n \times n$ random binary matrix with i.i.d.\ entries $m_{ij}$ such that $\prob(m_{ij} = 1) = p$ and $\prob(m_{ij} = 0) = 1-p$. Then, for any $\epsilon > 0$, $\prob(M \text{ has at least two equal rows or columns}) \to 0$ as $n \to \infty$ if $p \geq \frac{(1+\epsilon)\log(n)}{n}$.
\end{lemma}

\begin{proof}[Proof of Theorem \ref{alg}]
There are two cases in which we proceed to the second step of the algorithm: first, when there are at least two identical sub-weight vectors, or second, when at least two columns are identical. The probability of the second criteria is shown by Lemma \ref{duplicate} to converge to 0 as $n \to \infty$ for $p$ of the form described, so it suffices to show that the probability of the first criteria occurring also converges to 0 as $n \to \infty$. We call this event $A(n,p)$. Recall from section 3 that we begin step one of the algorithm by partitioning the columns according to their weight into collections $\mathcal{P} = (\mathcal{P}_1,...,\mathcal{P}_n)$, and that $I_k$ denotes the indices corresponding to columns in $\mathcal{P}_k$. 
    
For a particular $k$, let the sub-strings of $\rho_1$ and $\rho_2$ that only contains entries with indices in $I_k$ be denoted by $X = (X_1,...,X_{|\mathcal{P}_k|})$ and $Y=(Y_1,...,Y_{|\mathcal{P}_k|})$. In order to have $t_{1,k}=t_{2,k}$, we require that $\sum_{i=1}^{|\mathcal{P}_k|}X_i = \sum_{i=1}^{|\mathcal{P}_k|}Y_i$. The sums are equal if and only if
$$|\{ i : 1 \leq i \leq n, \ (X_i,Y_i) = (0,1)\}| = |\{i : 1 \leq i \leq n, \ (X_i,Y_i) = (1,0)\}| , $$
as an outcome of $(0,0)$ or $(1,1)$ does not change the gap between the sum (for shorthand we write $\#(0,1)$ and $\#(1,0)$ to denote the two cardinalities). Since each of the $(X_i,Y_i)$ are pairs of row entries that both lie within columns of weight $k$, and the 1's are equally likely to be anywhere in each of the columns, we can see that for any $i \in 1,...,|\mathcal{P}_k|$,

\begin{align*}
\Prob((X_i,Y_i) = (0,1)) = \Prob((X_i,Y_i) = (1,0)) = \frac{k(n-k)}{n(n-1)} . 
\end{align*}
Since we assume that $k \in [\lfloor np \rfloor, \lfloor np \rfloor + \lfloor \sqrt{np} \rfloor]$ it holds that there is some $\alpha \in (0,1)$ such that

$$\frac{k(n-k)}{n(n-1)} \sim \frac{(np+\alpha\sqrt{np})(n(1-p)-\alpha\sqrt{np})}{n(n-1)} = p\left(\frac{n}{n-1}\right)\left(1 + \frac{\alpha}{\sqrt{np}}\right)\left(1-p-\frac{\alpha p}{\sqrt{np}}\right) , $$
and so $\Prob((X_i,Y_i) = (0,1)) = \Prob((X_i,Y_i) = (1,0)) = \Theta(p)$ (note that $np \to \infty$ by the assumptions on $p$). For each $m \in \{1,...,n\}$, the conditional probability $\Prob(t_{1,k} = t_{2,k} | \{|\mathcal{P}_k| = m\})$ is equal to
\begin{align*}
\sum_{i=0}^{\lfloor m/2 \rfloor} \Prob(\#(0,1) + \#(0,1) = 2i)\Prob\big(\{\#(0,1) = \#(1,0) = i\} \big| \{\#(0,1) + \#(1,0) = 2i\}\big) . 
\end{align*}
Since $(0,1)$ and $(1,0)$ occur with equal probability, when we condition on there being $2i$ of them total, the values $\#(0,1)$ and $\#(0,1)$ follow a $\bin(2i,1/2)$ distribution. Define $\tilde{p} := \frac{2k(n-k)}{n(n-1)} = \Prob((X_i,Y_i) = (0,1)\text{ or } (1,0))$. From here, applying Stirling's approximation we obtain some $\beta > 0$ such that
\begin{align*}
\Prob(t_{1,k} = t_{2,k} | \{|\mathcal{P}_k| = m\}) &= \sum_{i=0}^{\lfloor m/2 \rfloor} \Prob(\bin(m,\tilde{p}) = 2i)\Prob(\bin(2i,1/2) = i)\\
&\leq  \beta\left(\sum_{i=0}^{\lfloor m/2 \rfloor} \frac{1}{\sqrt{2i \vee 1}}\Prob(\bin(m,\tilde{p}) = 2i)\right) \\
&\leq \beta\ex\left[ \frac{1}{\sqrt{\bin(m,\tilde{p}) \vee 1}} \right] \\
&\leq \frac{3\beta}{\sqrt{m\tilde{p}}} .
\end{align*}
See Lemma \ref{binomial} for a proof of the final inequality. Since we care about the case where $m \geq \gamma\sqrt{\frac{n}{p}}$ and take $n \to \infty$ we can safely assume the inequality holds. Let
$$S = \left\{\{(x_1,...,x_n) : x_i \geq \gamma\sqrt{\frac{n}{p}} \ \text{ for all } i \in [\lfloor np \rfloor , \lfloor np \rfloor + \lfloor \sqrt{np} \rfloor]\right\} , $$
where $\gamma > 0$ is the one from Lemma \ref{chernoff}. When we condition on the sizes of $|\mathcal{P}_k|$ for all $k \in \{1,...,n\}$, the events $\{t_{1,k} = t_{2,k}\}$ are all independent of each other and the sizes of all other columns. Hence, since increasing $m$ only decreases the upper bound for $\prob(t_{1,k} = t_{2,k}|\{|\mathcal{P}_k| = m\})$,
\begin{align*}
    &\sum_{(x_1,...,x_n) \in S} \Prob\left(T_1=T_2\Bigg| \bigcap_{k=1}^n \{|\mathcal{P}_k| = x_k\}\right)\Prob\left(\bigcap_{k=1}^n \{|\mathcal{P}_k| = x_k\}\right) \\
    \leq& \sum_{(x_1,...,x_n) \in S} \left(\frac{3\beta}{\sqrt{\gamma \tilde{p}\sqrt{\frac{n}{p}}}}\right)^{\sqrt{np}}\Prob\left(\bigcap_{k=1}^n\{|\mathcal{P}_k| = x_k\}\right) \\
    =& \left(\frac{3\beta}{\sqrt{\gamma \tilde{p}\sqrt{\frac{n}{p}}}}\right)^{\sqrt{np}}\Prob\bigg((|\mathcal{P}_1|,...,|\mathcal{P}_n|) \in S\bigg) \\
    \leq& \left(\frac{3\beta}{\sqrt{\gamma \tilde{p}\sqrt{\frac{n}{p}}}}\right)^{\sqrt{np}} ,
\end{align*}
where $T_i$ is the signature of $\rho_i$ as defined in section 3. On the other hand for $S^c$, we have that
\begin{align*}
\sum_{(x_1,...,x_n) \in S^c} \Prob\left(T_1=T_2\Bigg| \bigcap_{k=1}^n \{|\mathcal{P}_k| = x_k\}\right)\Prob\left(\bigcap_{k=1}^n \{|\mathcal{P}_k| = x_k\}\right) &\leq \Prob\bigg( (|\mathcal{P}_1|,...,|\mathcal{P}_n|) \in S^c \bigg) , 
\end{align*}
which is a good enough bound because Lemma \ref{chernoff} combined with the union bound ensures that the right side of the inequality is upper bounded by $(\sqrt{np})e^{-\frac{1}{6}\gamma\sqrt{\frac{n}{p}}}$. Putting these two pieces together we get that
\begin{equation}\label{anp}
    \Prob(A(n,p)) \leq n^2 \Prob(T_1=T_2)
    \leq n^2\left(\frac{3\beta}{\sqrt{\gamma \tilde{p}\sqrt{\frac{n}{p}}}}\right)^{\sqrt{np}} + n^2(\sqrt{np})e^{-\frac{1}{6}\gamma\sqrt{\frac{n}{p}}} . 
\end{equation}    
The right term clearly tends to 0 as $n \to \infty$. For the left term, we note that since $\tilde{p} = \Theta(p)$, we can group up all the constants into some $C > 0$ such that
\begin{align*}
    n^2\left(\frac{3\beta}{\sqrt{\gamma \tilde{p}\sqrt{\frac{n}{p}}}}\right)^{\sqrt{np}} &\leq n^2\left(\frac{C}{(np)^{1/4}}\right)^{\sqrt{np}} = \exp\left\{2\log(n) + 2\log(C)\sqrt{np} - \frac{1}{4}\sqrt{np}\log(np)\right\} ,
\end{align*}
which tends to 0 as $n \to \infty$ whenever $p \geq \frac{16(1+\epsilon)\log^2(n)}{n(\log\log(n))^2}$ for some $\epsilon > 0$. 

Now we discuss the time complexity of part two. As mentioned in Section 3, the time complexity of part two is $O(n^2P)$, where $P$ is the number of valid permutations to check. Hence, it is sufficient to show that $\E[P] = O(1)$ as part one always takes $O(n^2)$ time. The number of permutations we need to check only depends on the sizes of the sets of rows with the same sub-weight vectors, and not their positions. Thus, we sum over $j$ representing the number of non-unique sub-weight vectors, and then sum over $n_1,n_2,\dots,n_j$ such that $n_1+\dots+,n_j\leq n$, which represent the number of rows that share the same sub-weight vector. We also have the conditions $n_i>1$ as otherwise this would imply that it is a unique sub-weight vector, and $n_i\geq n_{i+1}$ as this avoids double counting. With this we get the following upper bounds for $\ex[P]$:

\begin{align*}
    &\sum_{j=1}^{n}\sum_{\substack{n_1+n_2+\dots+ n_j \leq n \\ 
    \forall i, n_i > 1 \quad n_i\geq n_{i+1}}} (n_1!)(n_2!)\dots(n_j!)\binom{n}{n_1,n_2,\dots, n_j}\prod_{i=1}^j\Prob(n_i \text{ rows have same sub-weight vector}) \\
    \leq& \sum_{j=1}^{n}\sum_{\substack{n_1+n_2+\dots +n_j \leq n \\ 
    \forall i, n_i > 1 \quad n_i\geq n_{i+1}}} (n_1!)(n_2!)\dots(n_j!)\binom{n}{n_1,n_2,\dots, n_j}\prob(T_1= T_2)^{\sum_{i=1}^j n_i-1} \\
    \leq& \left(1+\sum_{k=2}^n k! \binom{n}{k}\Prob(T_1=T_2)^{k-1}\right)^n .
\end{align*}

The last line, after expanding the product, contains terms which upper bound each term in the previous line upon applying the bound $\binom{n}{n_1,n_2,\dots, n_j}\leq \binom{n}{n_1}\dots\binom{n}{n_j}$. Reusing the bound from (\ref{anp}) we get that

\begin{align*}
    \prob(T_1=T_2) &\leq \left(\frac{C}{(np)^{1/4}}\right)^{\sqrt{np}} + (\sqrt{np})e^{-\frac{1}{6}\gamma\sqrt{\frac{n}{p}}} \leq n^{-(1+o(1))3\sqrt{1+\epsilon}},
\end{align*}
when $p\geq \frac{36(1+\epsilon)\log^2(n)}{n (\log\log(n))^2}$. Combining this with the above approximation for $\ex[P]$ we see that
\begin{align*}
    \E[P] &\leq \Bigg(1+\sum_{k=2}^n k!\binom{n}{k}\bigg(\frac{1}{n^{(1+o(1)3\sqrt{1+\epsilon}}}\bigg)^{k-1}\Bigg)^n \\
    &\leq \Bigg(1+\sum_{k=2}^n n^k\bigg(\frac{1}{n^{(1+o(1))2\sqrt{1+\epsilon}}}\bigg)^{k-1}\Bigg)^n \\
    &= \Bigg(1+n\sum_{k=1}^{n-1} \bigg(\frac{1}{n^{(1+o(1))2\sqrt{1+\epsilon}}}\bigg)^k\Bigg)^n \\
    &\leq \Bigg(1+n\left(\frac{1}{1-n^{-(1+o(1))2\sqrt{1+\epsilon}}}-1\right)\Bigg)^n \\
    &\leq \exp\left\{ \frac{n^2}{n^{(1+o(1))2\sqrt{1+\epsilon}}} \right\} \to 1,
\end{align*}
as $n \to \infty$. Hence $\ex[P] \to 1$ as $n \to \infty$ as there is always at least one valid permutation (the original ordering before shredding). 

\end{proof}

\section{Unique Reconstructibility}

A common problem of interest in most reconstruction models is that of finding which parameters $p = p(n)$ are such that reconstructibility of the structure being studied is guaranteed with high probability. Our algorithm gives an upper bound of $\frac{16\log^2(n)}{n (\log\log(n))^2}$ for the critical value at which reconstructibility can be ensured, though with the first moment method approach we can improve that bound.

\begin{theorem}\label{thresh}
    Let $M$ be an $n \times n$ random binary matrix with i.i.d.\ entries $m_{ij}$ with $\prob(m_{ij} = 1) = p$ and $\prob(m_{ij} = 0) = 1-p$. Then, for any $\epsilon > 0$, $\Prob(M \text{ is reconstructible}) \to 1$ as $n \to \infty$ for $p \geq \frac{2(1+\epsilon)\log(n)}{n}$.    
\end{theorem}

The following lemma offers us a second equivalent definition for reconstructibility that is better for completing the computations in the proof of Theorem \ref{thresh}

\begin{lemma}\label{definition}
    Let $M$ be an $n \times n$ binary matrix with shredded column and row collections given by $\gamma_1,...,\gamma_n$ and $\rho_1,...,\rho_n$ respectively, and let $M_{\sigma,\tau}$ denote the matrix obtained from permuting the rows by $\sigma$ and the columns by $\tau$, $M_{\sigma,\tau} = (m_{\sigma(i),\tau(j)})_{i,j=1}^n$ for a particular pair $(\sigma,\tau) \in S_n^2 \setminus\{(\id,\id)\}$ (here $\id$ just means the identity permutation that sends each $i \in [n]$ to itself). Then,
    $$\{M \text{ is not reconstructible}\} = \bigcup_{\substack{(\sigma,\tau) \in S_n^2 \\ (\sigma,\tau) \ne (\id,\id)}} \{ M_{\sigma,\tau} = M \} . $$
\end{lemma}

\begin{proof}[Proof of Theorem \ref{thresh}] Define,
$$N = \sum_{(\sigma,\tau) \in (S_n\setminus\{\id\})^2} \1_{\{M_{\sigma,\tau} = M\}} . $$ 
A quick computation shows that $\E[N] = (n!-1)^2\Prob(M_{\sigma,\tau} = M)$, where $(\sigma,\tau)$ are independent and both uniform over $S_n \setminus \{\id\}$. Before bounding this expression, we need some further exploration of the events $\{M_{\sigma,\tau} = M\}$.

We define the permutation graph of a pair $\sigma,\tau \in S_n$ to be the directed graph on $[n]^2 = \{(i,j) : 1 \leq i,j \leq n\}$ where each vertex $(i,j)$ has an out-going edge pointing to $(\sigma(i), \tau(j)) = (\sigma_i,\tau_j)$. If $\sigma,\tau \in S_n$ have cyclic decompositions $\sigma = a_1 \cdots a_m$ and $\tau = b_1 \cdots b_k$, a particular pair of cycles $a_i$ and $b_j$ acts on exactly the $|a_i| \times |b_j|$ sub-matrix of $M$ that corresponds to the rows that $a_i$ acts on and the columns that $b_j$ acts on (here $|\cdot|$ denotes the length). In the permutation graph, this $|a_i| \times |b_j|$ sized region corresponds exactly to a collection of $\gcd(|a_i|,|b_j|)$ disjoint cycles, all of length $\lcm(|a_i|,|b_j|)$.  In order to have $M_{\sigma,\tau} = M$, it is necessary to have all equality between all entries in $M$ that exist within the same cycle in the permutation graph. That is,
\begin{equation}\label{containment}
\{M_{\sigma,\tau} = M\} \subseteq \bigcap_{i,j=1}^n\{m_{i,j} = m_{\sigma^\ell(i),\tau^\ell(j)} \text{ for all }\ell \in \N\} .     
\end{equation}
Since each of the cycles are disjoint, the events in the intersection are all independent. Using (\ref{containment}) along with our original expression for $\ex[N]$ we get that
\begin{align*}
\ex[N] &\leq (n!)^2\ex\left[ \prod_{(i,j) \in S} \left(p^{\lcm(|a_i|,|b_j|)}+(1-p)^{\lcm(|a_i|,|b_j|)}\right)^{\gcd(|a_i|,|b_j|)}\right] ,
\end{align*}
where $S = \{(i,j) : 1 \leq i \leq m, 1 \leq j \leq k, (|a_i|,|b_j|) \ne (1,1)\}$. If we let $c_1(\sigma)$ and $c_1(\tau)$ denote the number of singleton cycles in $\sigma$ and $\tau$, then we can factor out powers of $(1-p)$ and use the fact that $|a_i|\cdot|b_j| \geq 2$ for $(i,j) \in S$,
\begin{align*}
\ex[N] &\leq (n!)^2\ex\left[(1-p)^{n^2-c_1(\sigma)c_1(\tau)}\prod_{(i,j) \in S} \left(1 + \left(\frac{p}{1-p}\right)^{\lcm(|a_i|,|b_j|)}\right)^{\gcd(|a_i|,|b_j|)}\right] \\
&\leq (n!)^2\ex\left[ (1-p)^{n^2-c_1(\sigma)c_1(\tau)}\exp\left\{\sum_{(i,j) \in S}\left(\frac{p}{1-p}\right)^{2}\right\}\right] \\
&\leq (n!)^2\ex\left[e^{-pn^2+pc_1(\sigma)c_1(\tau)}e^{4(n^2-c_1(\sigma)c_1(\tau))p^2}\right] .
\end{align*}
By bounding the expected value in the final upper bound, one can show that $\ex[N] \to 0$ as $n \to \infty$ for
$$\frac{(2+\epsilon)\log(n)}{n} \leq p \leq \frac{17\log^2(n)}{n(\log\log(n))^2},$$
which is sufficient because Theorem 1 covers the case where $p \geq \frac{16(1+\epsilon)\log^2(n)}{n(\log\log(n))^2}$ for any $\epsilon > 0$. Applying Lemma \ref{definition} with the union bound we see that
$$\prob(M \text{ is reconstructible}) \leq \ex[N] + \prob\left( 
\bigcup_{\substack{(\sigma,\tau) \in S_n^2\setminus(\id,\id) \\ \sigma = \id \text{ or } \tau = \id}} \{M_{\sigma,\tau} = M\} \right) .$$
However, if one of $\sigma$ or $\tau$ are the identity there must be at least two rows or columns that are identical in $M$ as the other cannot be the identity. Thus by Lemma \ref{duplicate}
$$\prob\left( 
\bigcup_{\substack{(\sigma,\tau) \in S_n^2\setminus(\id,\id) \\ \sigma = \id \text{ or } \tau = \id}} \{M_{\sigma,\tau} = M\} \right) \to 0 \ \text{as $n \to \infty$ for $p \geq \frac{2(1+\epsilon)\log(n)}{n}$} .$$
Combining this with the above we obtain that $\prob(M \text{ is not reconstructible}) \to 0$ as  $n \to \infty$ for $p \geq \frac{2(1+\epsilon)\log(n)}{n}$.
\end{proof}

\section{Proofs of Lemmas}

\setcounter{theorem}{1}

\begin{lemma}
Suppose that $p=p(n)$ is some sequence such that $np \geq 16$. There exists a positive constant $\gamma > 0$ such that $b_{n,p,\lfloor np \rfloor +i} \geq 2\gamma\frac{1}{\sqrt{np}}$ for all $i \in [0,\lfloor\sqrt{np}\rfloor]$. Furthermore,
$$\Prob\left(|\mathcal{P}_{\lfloor np \rfloor+i}| \leq \gamma\sqrt{\frac{n}{p}}\right) \leq e^{-\frac{1}{6}\gamma \sqrt{\frac{n}{p}}} , $$
where $\mathcal{P}_i = \{1 \leq j \leq n : |\gamma_j| = i\}$ and $\gamma_1,...,\gamma_n$ are the shredded columns.
\end{lemma}

\begin{proof}
Since $|\mathcal{P}_k| = \sum_{i=1}^n \1_{\{\text{column $i$ has weight $k$}\}}$, and each column has weight $k$ with probability $b_{n,p,k} = \binom{n}{k}p^k(1-p)^{n-k}$, it holds that $|\mathcal{P}_k| \sim \bin(n,b_{n,p,k})$. By a Chernoff bound we obtain,
$$\Prob\left(|\mathcal{P}_k| \leq \frac{1}{2}nb_{n,p,k} \right) \leq e^{-\frac{1}{12}nb_{n,p,k}} . $$
From here it suffices to show that there is a constant $\gamma > 0$ such that $\frac{1}{2}nb_{n,p,k} \geq \gamma\sqrt{\frac{n}{p}}$ when $k = \lfloor np \rfloor + i$, $i \in [0 , \lfloor \sqrt{np} \rfloor]$. To do this we show the following: for any $0 \leq x \leq \sqrt{np}$ such that $np+x$ is integer-valued, $b_{n,p,np+x} \geq \frac{\alpha}{\sqrt{np}}$ for some $\alpha > 0$. Repeatedly apply Stirling's bounds
$$\sqrt{2\pi n}\left(\frac{n}{e}\right)^n \leq n! \leq e\sqrt{2\pi n}\left(\frac{n}{e}\right)^n , $$
to yield $b_{n,p,np+x} \geq e^{-2} A_1A_2A_3$, where
$$A_1 = \frac{1}{\left(1+\frac{x}{np}\right)^{np}\left(1-\frac{x}{n(1-p)}\right)^{n(1-p)}}, \ A_2 = \left(\frac{1-\frac{x}{n(1-p)}}{1+\frac{x}{np}}\right)^x, A_3 = \frac{1}{\sqrt{2\pi}}\sqrt{\frac{n}{(np+x)(n(1-p)-x)}} . $$
Using the fact that $1+y \leq e^y$ for all $y \in \R$, we get
$$A_1 \geq \frac{1}{e^{x}e^{-x}} = 1 .$$
Next, since $p \geq \frac{16}{n}$,
$$A_2 \geq \left( \left(1-\frac{x}{n(1-p)}\right)\left(1-\frac{x}{np}\right)\right)^x \geq \left(1-\frac{x}{np(1-p)}\right)^x \geq \left(1-\frac{2}{\sqrt{np}}\right)^{\sqrt{np}} \geq \frac{1}{2^4} = \frac{1}{16} .$$
Finally, using again the fact that $\frac{1}{2} \geq p \geq \frac{16}{n}$,
\begin{align*}
A_3 &\geq \frac{1}{\sqrt{2\pi n p(1-p)}} \frac{1}{\sqrt{(1+\frac{x}{np})(1-\frac{x}{n(1-p)})}} \\
&\geq \frac{1}{\sqrt{2\pi n p(1+\frac{x}{np(1-p)}})} \\
&\geq \frac{1}{\sqrt{2\pi np (1 + \frac{2}{\sqrt{np}})}} \\
&\geq \frac{1}{\sqrt{3\pi np}}.    
\end{align*}
Putting everything together we get that
$$b_{n,p,np+x} \geq \left(\frac{1}{16e^2\sqrt{3\pi}}\right)\frac{1}{\sqrt{np}}.$$
\end{proof}

\begin{lemma}
Let $M$ be an $n \times n$ random binary matrix with i.i.d.\ entries $m_{ij}$ such that $\prob(m_{ij} = 1) = p$ and $\prob(m_{ij} = 0) = 1-p$. Then, for any $\epsilon > 0$, $\prob(M \text{ has at least two equal rows or columns}) \to 0$ as $n \to \infty$ if $p \geq \frac{(1+\epsilon)\log(n)}{n}$.
\end{lemma}

\begin{proof}
Let $r_1,...,r_n$ be the rows of $M$, $A_{i,j} = \{r_i =r_j\}$, and let $N = \sum_{i \ne j} \1_{A_{i,j}}$. Then,
\begin{align*}
\E[N] &= \binom{n}{2}\Prob(A_{1,2}) = \binom{n}{2}(p^2+(1-p)^2)^n \\
&= \binom{n}{2}(1-p)^{2n}\left(1 + \frac{p^2}{(1-p)^2}\right)^n \\
&\leq \binom{n}{2}e^{-2np+4np^2} \to 0 \quad \text{as $n \to \infty$},
\end{align*}
when $p \geq (1+\epsilon)\frac{\log(n)}{n}$. Since the columns have the same distribution as the rows, the result follows by Markov's inequality.
\end{proof}

\setcounter{theorem}{4}

\begin{lemma}
    Let $M$ be an $n \times n$ binary matrix with shredded column and row collections given by $\gamma_1,...,\gamma_n$ and $\rho_1,...,\rho_n$ respectively, and let $M_{\sigma,\tau}$ denote the matrix obtained from permuting the rows by $\sigma$ and the columns by $\tau$, $M_{\sigma,\tau} = (m_{\sigma(i),\tau(j)})_{i,j=1}^n$ for a particular pair $(\sigma,\tau) \in S_n^2 \setminus\{(\id,\id)\}$ (here $\id$ just means the identity permutation that sends each $i \in [n]$ to itself). Then,
    \begin{equation}\label{deflem}
    \{M \text{ is not reconstructible}\} = \bigcup_{\substack{(\sigma,\tau) \in S_n^2 \\ (\sigma,\tau) \ne (\id,\id)}} \{ M_{\sigma,\tau} = M \} .         
    \end{equation}
\end{lemma}

\begin{proof}
Suppose that $M$ is not reconstructible. Then, there exists two distinct pairs of permutations $(\sigma,\tau)$ and $(\sigma',\tau')$ that satisfy (\ref{eq:recon}). That is, there is some matrix $M'$ (possibly both equal to $M$) such that
\begin{align*}
M
= 
\begin{bmatrix}
    \rho_{\sigma_1} \\ \vdots \\ \rho_{\sigma_n}
\end{bmatrix}
= 
\begin{bmatrix}
\gamma_{\tau_1} \ \cdots \ \gamma_{\tau_n}    
\end{bmatrix}
, \text{ and }
M' =
\begin{bmatrix}
    \rho_{\sigma'_1} \\ \vdots \\ \rho_{\sigma'_n}
\end{bmatrix}
= 
\begin{bmatrix}
\gamma_{\tau'_1} \ \cdots \ \gamma_{\tau'_n}    
\end{bmatrix}
.
\end{align*}
Suppose that $r_1,...,r_n$ and $c_1,...,c_n$ are the rows and columns in their original, pre-shredding order. Applying $\sigma' \circ \sigma^{-1}$ to $(r_1,...,r_n)$ and $\tau \circ {(\tau{'})^{-1}}$ to $(c_1,...,c_n)$ must necessarily send $M$ back to itself: 
\begin{enumerate}
    \item Applying $\sigma^{-1}$ to the rows of $M$ yields the matrix associated with the shredded rows $R = [\rho_1 \cdots \rho_n]^T$;
    \item Applying $\sigma'$ to $R$ gives $[\rho_{\sigma'_1} \cdots \rho_{\sigma'_n}]^T = [\gamma_{\tau'_1} \cdots \gamma_{\tau'_n}] = M'$ by the above identity;
    \item Applying $(\tau')^{-1}$ to $M'$ brings us to the matrix associated with the shredded columns $C = [\gamma_1 \cdots \gamma_n]$;
    \item Finally, applying $\tau$ to $C$ brings us back to our original matrix $[\gamma_{\tau_1} \cdots \gamma_{\tau_n}] = M$.
\end{enumerate} 
Both of these two permutation pairs cannot be the identity because we assume the pairs are distinct, and so the inclusion $\subseteq$ holds in (\ref{deflem}).

For the other direction in (\ref{deflem}) we suppose we have $(\sigma,\tau) \in S_n^2 \setminus{(\id,\id)}$ such that $M_{\sigma,\tau} = M$. Then if we are given an arbitrary shredded ordering $\gamma_1,...,\gamma_n$ and $\rho_1,...,\rho_n$, we can by assumption always apply $(\sigma,\tau)$ to the correct ordering to obtain a new non-equal ordering that is valid. Hence, $M$ is not reconstructible.
\end{proof}

\begin{lemma}\label{binomial}
Suppose that $X \sim \bin(n,p)$ and that $np \geq 16$. Then,
$$\ex\left[\frac{1}{\sqrt{X \vee 1}}\right] \leq \frac{3}{\sqrt{np}} . $$
\end{lemma}

\begin{proof}
Splitting the expectation into two pieces and then applying the Chebyshev-Cantelli inequality gives us the upper bound
\begin{align*}
    \ex\left[ \frac{1}{\sqrt{X \vee 1}} \right] \leq \sqrt{\frac{2}{np}} + \Prob\left( X \leq \frac{np}{2} \right) \leq \sqrt{\frac{2}{np}} + \frac{np(1-p)}{np(1-p)+(\frac{np}{2})^2} .
\end{align*}
Utilizing the fact that $np \geq 16$ we can see that
$$\frac{np(1-p)}{np(1-p)+(\frac{np}{2})^2} \leq \frac{4}{4+np} \leq \frac{1}{\sqrt{np}} , $$
which combined with the above gives
$$\ex\left[\frac{1}{\sqrt{X \vee 1}}\right] \leq \frac{\sqrt{2}}{\sqrt{np}} + \frac{1}{\sqrt{np}} \leq \frac{3}{\sqrt{np}}.$$
\end{proof}

\begin{lemma}\label{final}
Let $\sigma,\tau$ be independent uniform permutations over $S_n \setminus \{\id\}$, and let $c_1(\sigma),c_1(\tau)$ be the number of singleton cycles in both $\sigma$ and $\tau$ respectively. Then, for any $\epsilon > 0$,
$$a_n := (n!)^2\ex\left[e^{-pn^2+pc_1(\sigma)c_1(\tau)}e^{4(n^2-c_1(\sigma)c_1(\tau))p^2}\right] \to 0$$
as $n \to \infty$ for 
\begin{equation}\label{condition}
\frac{(2+\epsilon)\log(n)}{n} \leq p \leq \frac{17\log^2(n)}{n(\log\log(n))^2}.    
\end{equation}
\end{lemma}

\begin{proof}
First, we write the expression in the statement of the lemma as
\begin{equation*}
a_n = \sum_{0 \leq x,y \leq n-1} (n!)^2\exp\left\{-pn^2+pxy+4(n^2-xy)p^2\right\}\prob(c_1(\sigma)=x)\prob(c_1(\tau)=y) .   
\end{equation*}
We split off the terms with $xy = 0$ and upper bound by
\begin{equation*}
a_n \leq 2n(n!)^2\exp\left\{-pn^2+4n^2p^2\right\} + C\sum_{1 \leq x,y \leq n-1} \frac{(n!)^2}{x!y!}\exp\left\{-pn^2+pxy+4(n^2-xy)p^2\right\} ,
\end{equation*}
for some $C > 0$ such that $\prob(c_1(\sigma) = x) \leq  \sqrt{C}\frac{1}{x!}$. Such a $C$ exists because $\prob(c_1(\sigma') = x) \sim \frac{1}{x!}$ for $\sigma' \in S_n$ uniformly drawn (see \textcite{Arratia1992,Ford2022} for a discussion of random permutation statistics). One can see immediately that the first term tends to $0$ for $p$ in the described range, so all that is left is the second term. Relabelling $x = n-k$ and $y = n-\ell$ we can upper bound the sum by
\begin{equation}\label{finalsum}
C\sum_{1 \leq k,\ell \leq n-1} n^{k+\ell}\exp\left\{ -p(n(k+\ell)-k\ell)(1+o(1))\right\} \leq C\sum_{1 \leq \ell \leq n-1}\bigg(n\sup_{0 \leq k \leq n} f_\ell(k)\bigg),
\end{equation}
where $f_\ell(k) = e^{-((np-\log(n))(k+\ell)-pk\ell)(1+o(1))}$ with $k$ now being allowed to take on real values. To find $\max_{0 \leq k \leq n} f_\ell(k)$ it suffices to find $\min_{0 \leq k \leq n} ((np-\log(n))(k+\ell)-pk\ell) := \min_{0 \leq k \leq n} g_\ell(k)$. Since $g_\ell(k)$ is linear in $k$ it is monotone, and so
\begin{align*}
\min_{0 \leq k \leq n} g_\ell(k) &= \min\bigg\{(np-\log(n))\ell,(n^2p-n\log(n)-\ell\log(n))\bigg\} \geq \min\bigg\{(1+\epsilon)\log(n)\ell,\epsilon n \log(n)\bigg\}.
\end{align*}
For the above inequality we use the assumptions on $p$ from (\ref{condition}). Combining this bound with (\ref{finalsum}) gives 
$$a_n \leq C\sum_{1 \leq \ell \leq n-1} n^{-\epsilon \ell(1+o(1))} + C\sum_{1  \leq \ell \leq n-1} n^{-n\epsilon(1+o(1))-1} \leq C\sum_{1 \leq \ell \leq n-1} n^{-\epsilon \ell(1+o(1))} + Cn^{-n\epsilon(1+o(1))-2},$$
which converges to 0 as $n \to \infty$. Altogether, this proves that $a_n \to 0$ for $p$ in the desired range.
\end{proof}

\printbibliography

\end{document}